\documentclass[a4paper,12pt]{article}
\usepackage{amsmath,amssymb,amsfonts,amsthm,amstext,amsbsy,amsopn,amscd}
\usepackage{newlfont}
\usepackage{lscape}
\usepackage{enumerate}
 \newcommand{\dist}{\mathrm{d} }
 \newcommand\calP{{\mathcal P}}
 \newtheorem{thm}{Theorem}[section]
 \newtheorem{cor}[thm]{Corollary}
 \newtheorem{lem}[thm]{Lemma}

 \theoremstyle{definition}
 \newtheorem{defn}[thm]{Definition}
 \newtheorem{rem}[thm]{Remark}

 \newcommand{\diam}{\mathrm{diam}}
 
 \newcommand{\Ga}{\Gamma}
 \newcommand{\Si}{\Sigma}

 \newcommand{\Aut}{\mathrm{Aut}}

 \newcommand{\PGaL}{\mathrm{P\Gamma L}}

 \newcommand{\PGL}{\mathrm{PGL}}
 \newcommand{\PSL}{\mathrm{PSL}}
 \newcommand{\PSU}{\mathrm{PSU}}
 
 \newcommand{\PSp}{\mathrm{PSp}}
 \newcommand{\GF}{\mathrm{GF}}
 \newcommand{\PG}{\mathrm{PG}}
 \newcommand{\G}{\mathrm{G}}
 
 \newcommand{\Inc}{\mathrm{Inc}}
 \def\S{\textsf{S}}
 \def\ssig{\S(\Sigma)}

 \def\HoSi{\mathrm{HoSi}}
 \def\Petersen{\mathrm{Petersen}}

\title{A classification of graphs whose subdivision graphs are locally $G$-distance
transitive}
\author{Ashraf Daneshkhah and Alice Devillers}

\begin{document}
\maketitle

\begin{abstract}
The subdivision graph $\ssig$ of a connected graph $\Sigma$ is constructed by adding
a vertex in the middle of each edge. 
In a previous paper written with Cheryl E. Praeger, we characterised the graphs $\Sigma$ such that $\ssig$ is locally $(G,s)$-distance transitive for $s\leq 2\, \diam(\Sigma)-1$ and some $G\leq \Aut(\Sigma)$.
In this paper, we solve the remaining cases by classifying all
the graphs $\Sigma$ such that the subdivision graphs is locally
$(G,s)$-distance transitive for $s\geq 2\, \diam(\Sigma)$ and some $G\leq \Aut(\Sigma)$. In particular,  their subdivision graph are always locally 
$G$-distance transitive, except for the complete graphs.
\end{abstract}

\section{Introduction}

In this paper all graphs are simple, undirected and connected, that is to say, a
graph $\Sigma$ consists of a  vertex-set $V\Sigma$ and a subset  $E\Sigma$ of
unordered pairs from $V\Sigma$, called edges, and for any two vertices $u,v$ there is a
`path' $\{u_0,u_1,\dots,u_m\}$ with each $\{u_i,u_{i+1}\}$ an edge and $u_0=u,
u_m=v$.

The {\it subdivision graph} $\S(\Sigma)$ of a graph $\Sigma$  is defined as the
graph with vertex set $V\Sigma \cup E\Sigma$ and edge set
\begin{align*}
  \{\{x, e\}| \ x\in V\Sigma, \  e\in E\Sigma, \ x\in e\}.
\end{align*}
Informally, $\ssig$ is the graph obtained from $\Sigma$ by `adding a vertex' in the
middle of each edge of $\Sigma$. Note that $\ssig$ is bipartite with bipartition
$V\Sigma$ and $E\Sigma$. We have proved in \cite{DDP2009} that the diameter of $\ssig$ is equal to $2\, \diam(\Sigma)+\delta$ for some $\delta\in \{0,1,2\}$.
Observe that $\Aut(\Si)$ can be identified with a subgroup of $\Aut(\ssig)$, and we proved in \cite{DDP2009} that, with the exception of $\Sigma$ being a cycle, we have   $\Aut(\Si)=\Aut(\ssig)$.
From now on, we will always consider groups $G\leq \Aut(\Sigma)$, acting on $\Si$ as well as on $\ssig$.
Various symmetric properties of subdivision graphs were studied in \cite{DDP2009}, in particular their local $(G,s)$-distance transitivity.

Let $\Ga$ be a graph with diameter $D$, and let $G$ be a subgroup of $\Aut
\Ga$.
We denote by $\Ga_i(x)$ the set of vertices of $\Ga$ at distance $i$ from $x$.
 For $1\leqslant s  \leqslant D$, we say that $\Ga$ is \emph{locally
$(G,s)$-distance transitive} if for all $x\in V\Ga$, $G_x$ is transitive on
$\Ga_i(x)$  for  $1\leqslant i\leqslant s$, and  we say that $\Ga$ is
\emph{locally  $G$-distance transitive} if it is locally
$(G,D)$-distance transitive. It is known that, for a
locally $(G,1)$-distance transitive graph $\Ga$, either $G$ is transitive on $V\Ga$, or $G$ has two orbits on vertices and the graph is bipartite (see Lemma 2.1 of \cite{DDP2009}). 

In \cite{DDP2009}, we characterised the graphs $\Sigma$ such that $\Ga=\ssig$ is locally $(G,s)$-distance transitive for $s\leq 2\, \diam(\Sigma)-1$. 
We started exploring the case where $s\geq  2\, \diam(\Sigma)$, and classified all such graphs with $s$ at most 5. 
In this paper, we extend this result to all $s$.

We now state  the main result of this paper.

\begin{thm}\label{char}
Let $\Sigma$ be a connected graph of diameter $d\geq 2$, and let $G\leq\Aut(\Sigma)$ . Let also  $\ssig$ be the subdivision graph of $\Si$. Then the following are equivalent:
\begin{enumerate}
\item[(a)] $\ssig$ is locally $(G,2d)$-distance transitive;
\item[(b)] $\Sigma$ and $G$ are as in the first two columns of Table {\rm\ref{table1}};
\item[(c)] $\ssig$ is locally $G$-distance transitive.
\end{enumerate}
\end{thm}

\begin{table}[h]\label{table1}
\caption{$(\Si,G)$ such that $\ssig$ is locally $(G,2d)$-distance transitive}
\begin{tabular}{|c||p{2.5cm}|p{6cm}||l|c|c|c|}
\hline
&$\Si$&$G$&$|V\Si|$&$g$&$d$&$D$\\

\hline\hline
1&$K_{n,n}$, $n\geq 3$& $G\leq S_n\wr S_2$ satisfies Condition (*)&$2n$&$4$&$2$&$4$\\
\hline
2&$\Petersen$&$S_5$&$10$& $5$&$2$&$6$\\
\hline
3&$\HoSi$&$\PSU_3(5)$ or $\PSU_3(5).2$&$50$& $5$&$2$&$6$\\
\hline

4&$\Inc(\PG(2,q))$ &$\PSL(3,q)\leq G\leq \Aut (\PSL(3,q))$ and $G$ contains a duality&$ 2(\frac{q^3-1}{q-1})$&$6$&$3$&$6$\\

\hline

5&$\Inc(W(3,q))$, $q=2^m$&$\PSp(4,q)\leq G\leq \Aut
(\PSp(4,q)) $ and $G$ contains a duality&$ 2(\frac{q^4-1}{q-1})$ &$8$&$4$&$8$\\
\hline
6&$\Inc(W(3,2))$&$M_{10}$&$30$ &$8$&$4$&$8$\\
\hline

7&$\Inc(H(q))$,  $q=3^m$&$\G_2(q)\leq G \leq
\Aut(\G_2(q))$and $G$ contains a duality &$2(\frac{q^6-1}{q-1})$&$12$&$6$&$12$\\
\hline
8&$C_n$&$D_{2n}$&$n$&$n$&$\lfloor n/2\rfloor$&$n$\\
\hline

\end{tabular}
\end{table}

\underline{Comments on Table \ref{table1}}:
\begin{itemize}
 \item[(a)] The last 4 columns give more information on the graph $\Si$, namely its number of vertices, its girth $g$, its diameter $d$ and the diameter $D$ of $\ssig$. 
\item[(b)] We denote by $\Inc({\mathcal G})$, the incidence graph of the corresponding rank 2 geometry ${\mathcal G}$, which is a bipartite graph (points/lines). 
We have the following geometries (which are all generalised polygons): $\PG(2,q)$ is the Desarguesian projective plane over $\GF(q)$,  $W(3,q)$ is the symplectic generalised quadrangle  over $\GF(q)$,
 and  $H(q)$ is the classical hexagon over $\GF(q)$. 
Notice that $\Inc(W(3,2))$ is also known as the Tutte-Coxeter graph or Tutte eight-cage. 
\item[(c)] 
For the sake of completeness, we reproduce here Condition (*), which was determined in Example 5.3 of \cite{DDP2009}.
We say that $G\leq S_n\wr S_2$ satisfies Condition (*) if and only if 
       \begin{itemize}
            \item[(i)] $G\leqslant H\wr S_2$ with component $H$, where $H$ is 2-transitive on each orbit $\Delta_1$ and $\Delta_2$;
               \item[(ii)] for $u_1\in\Delta_1$, $G_{u_1}$ is transitive 
               on $(\Delta_1\setminus\{u_1\})\times \Delta_2$; and
                \item[(iii)] for $u_1\in\Delta_1$ and $u_2\in\Delta_2$, the stabiliser $G_{\{u_1,u_2\}}$ interchanges
          $u_1$ and $u_2$, and is transitive
              on $\{\{v_1,v_2\}\,|\, v_i\in \Delta_i\setminus\{u_i\}\,\}$. 
         \end{itemize}
\end{itemize}

\begin{rem}
 If $\Si$ has diameter $1$, then $\Si$ is a complete graph $K_n$, and this case was completely treated in \cite{DDP2009}. In particular, $\S(K_n)$ is  locally $(G,2)$-distance transitive 
if and only if $G$ is $3$-transitive on $V\Si$ (or $G=S_2$ if $n=2$). For $n=2$ or $3$, $\S(K_n)$ is  locally $(G,2)$-distance transitive if and only if it is  locally $G$-distance transitive. 
However, for $n\geq 4$, $\S(K_n)$ is locally $G$-distance transitive if and only if $G$ is $4$-transitive on $V\Si$, or $n=9$ and $G=\PGaL(2,8)$. 
Because the conditions on $G$ for $\ssig$ to be locally $(G,2)$-distance transitive or locally $G$-distance transitive are different, we excluded this case from  Theorem \ref{char}. 
\end{rem}

The main tool for proving Theorem \ref{char} is Lemma \ref{basic}, where we prove that a connected graph of valency $k\geq 3$, girth $g$ and diameter $d$ whose subdivision graph is locally $(G,2d)$-distance transitive is a $(k,g)$-cage (see Section 2 for definition).


In \cite{DDP2009}, we proved that if  $\ssig$ is locally $(G,s)$-distance transitive  for $2d\leq s \leq D=2d+\delta$ and if $\Si$ has valency at least 3, then $s$ is at most $14+\delta$.  
The following Corollary \ref{diamsSub}, which follows immediately from Theorem\ref{char} and Table\ref{table1}, improves this result. 
\begin{cor}\label{diamsSub}
 Let $\Sigma$ be a connected graph of valency at least $3$ and diameter $d$  at least $2$.  Let $G\leq\Aut(\Sigma)$, and suppose that  $\ssig$ is 
 locally $(G,s)$-distance transitive  for $2d\leq s \leq D=2d+\delta$. Then  $d\leq 6 $ and $s\leq D \leq 12$.
\end{cor}


\section{Preliminary results}

Before we prove our main theorem, we are going to state a few more definitions, and prove some necessary Lemmas.

The \emph{line graph} $L(\Gamma)$ of  a graph $\Gamma$ is defined as the graph with
vertex set $E\Ga$ and edges $\{e_1, e_2\}$, for $e_1, e_2\in E\Gamma$ such that
$e_1\cap e_2\neq \emptyset$. The \emph{distance $2$ graph} of $\Gamma$ is the graph
$\Gamma^{[2]}$ with the same vertex set as $\Gamma$ but with the edge set replaced
by the set of all vertex pairs $\{u,v\}$ such that $d_\Gamma(u,v)=2$.  If $\Gamma$
is connected and bipartite, then all vertices at even $\Gamma$-distance from $v$ lie in the same
connected component of $\Gamma^{[2]}$, and all vertices at odd $\Gamma$-distance
from $v$ lie in another connected component, and hence $\Gamma^{[2]}$  has two connected components. 

Note that for a connected graph $\Sigma$, by definition, the subdivision graph
$\ssig$ is bipartite, and so the distance $2$ graph of $\ssig$ has two connected components $\Sigma$ and
$L(\Sigma)$. The following Lemma gives more details about the connected
components of the subdivision graph $\ssig$ when $\ssig$ is locally $(G,2d)$-distance
transitive.

\begin{lem} \label{lem:line}
Let $\Sigma$ be a connected graph of diameter $d$, whose subdivision graph $\ssig$ has diameter $D$ satisfying $D\leq 2d+1$. If $\ssig$ is locally $(G,2d)$-distance transitive 
for $G\leq \Aut(\Sigma)$, then $\Sigma$ and $L(\Sigma)$ are $G$-distance transitive. 
\end{lem}
\begin{proof}
Since $\S(\Sigma)$ is locally $(G,2d)$-distance transitive, in particular it is locally
$(G,1)$-distance transitive. Note also that $G$ is intransitive on the vertices of $\ssig$. Therefore, by \cite[Lemma 2.1]{DDP2009}, $G$ has two orbits on the vertices of $\ssig$, 
namely one corresponding to the vertices of $\Sigma$ and one to the edges of $\Sigma$.
In other words, $L(\Sigma)$ and $\Sigma$ are $G$-vertex transitive. 
Note that the distance in $\ssig$ between two vertices in the same bipart (that is, either two vertices of $\Si$ or two edges of $\Si$) is even, and so is at most $2d$.
Thus $d_{\Sigma}(u,v)=\frac{1}{2}d_{\S(\Sigma)}(u,v)\leq d$ and
$d_{L(\Sigma)}(e,f)=\frac{1}{2}d_{\S(\Sigma)}(e,f)\leq d$, where $u,v\in V\Sigma$ and
$e,f\in E\Sigma$. Since
$\S(\Sigma)$ is locally $(G,2d)$-distance transitive,
 we have that the two graphs $\Sigma$ and $L(\Sigma)$ are locally $G$-distance transitive. Since they are also vertex-transitive, they are $G$-distance transitive.
\end{proof}

For a positive integer $s$, an \emph{$s$-arc} of a graph is an
$(s+1)$-tuple $(v_0,v_1,\ldots, v_s)$ of vertices such that $v_i$ is
adjacent to $v_{i-1}$ for all $1\leqslant i\leqslant s$ and $v_{j-1}\neq
v_{j+1}$ for all $1\leqslant j\leqslant s-1$. The number $s$ is called
the length of the $s$-arc. 
An $s$-arc such that $d(v_0,v_s)=s$ is called an \emph{$s$-geodesic}.
We say that $\Si$ is $(G,s)$-arc transitive for some $s$ if $G$ acts transitively on
the set of $s$-arcs.
Moreover  we say that $\Si$ is $(G,s)$-transitive if $G$ acts
transitively on the set of $s$-arcs but intransitively on the set of $(s+1)$-arcs. So  a finite $(G,s)$-arc transitive graph is  $(G,t)$-transitive for some $t\geq s$.

\begin{lem}\label{lem:d-arc}
Let $\Sigma$ be a connected graph of diameter $d$ and girth $g$, whose subdivision graph $\ssig$ is locally $(G,2d)$-distance transitive. Then $\Si$ is $(G,d)$-arc transitive and 
so is $(G,s)$-transitive for some $s\geq d$.
Moreover $2d\leq g \leq 2d+1$.
\end{lem}

\begin{proof}
Let $\ssig$ be locally $(G,2d)$-distance transitive. Then $\ssig$ is locally $(G,2d-1)$-distance transitive. So by \cite[Theorem 1.3]{DDP2009}, $\Si$ is $(G,d)$-arc transitive. 
Since $\diam(\Si)=d$, there exists $d$-arcs which are $d$-geodesics and so all $d$-arcs of $\Si$ are $d$-geodesics. It follows that $g\geq 2d$. 
On the other hand, if $g\geq 2d+2$, then $\Sigma$ would have vertices at distance greater than $d$, a contradiction. Therefore  $2d\leq g \leq 2d+1$.
\end{proof}

In \cite{Weiss85}, Weiss classified all the $(G,s)$-transitive graphs for $s\geq 4$. From his result, we can deduce the following theorem.
\begin{thm}\label{Weiss}
Let $s\geq 4$ and let $\Si$ be a connected graph of valency $k\geq 3$, and girth $g\leq 2s$.  Suppose that $\Si$ is $(G,s)$-transitive. Then $s\leq \frac{g+2}{2}$ and 
$\Si,s,g,k,G$ are as in one of the lines of Table~\ref{tab:Weiss}. 
\end{thm}

\begin{table}[h]\label{tab:Weiss}
\centering
\begin{tabular}{|l|l|l|l|l|}
\hline
$\Si$&$s$&$g$&$k$&$G$\\
\hline
$\Inc(\PG(2,q))$&$4$&$6$&$q+1$&$\PSL(3,q)\leq G\leq \Aut(\PSL(3,q))$\\
\hline
$\Inc(W(3,q))$, $q=2^m$&$5$&$8$& $q+1$&$\PSp(4,q)\leq G\leq \Aut
(\PSp(4,q)) $\\
\hline
$\Inc(W(3,2))$&$4$&$8$& $3$&$ G'=A_6$\\
\hline
$\Inc(H(q))$,  $q=3^m$& $7$&$12$&$q+1$& $\G_2(q)\leq G \leq \Aut(\G_2(q))$\\
\hline
$3$-fold cover of $\Inc(W(3,2))$& $5$&$10$& $3$& $3$-fold cover of $S_6$\\
\hline
\end{tabular}
\end{table}

The following bound is easy to calculate, See Theorem 1.1 in Chapter 6 of \cite{HandSh1993} for instance.

\begin{lem}\label{lemcage}
 Let $\Si$ be a regular graph of girth $g$ and valency $k$. Then $|V\Sigma|\geq n_0(k,g)$ where
\[
n_0(k,g)=
\left\{
  \begin{array}{ll}
    1+k+k(k-1)+\cdots +k(k-1)^{\frac{g-3}{2}} & \hbox{ if } g \hbox{ is odd } ;\\
   2(1+(k-1)+\cdots (k-1)^{\frac{g-2}{2}}) & \hbox{ if } g \hbox{ is even}.
  \end{array}
\right.
\]
\end{lem}

\begin{defn}
Let $\Si$ be a regular graph of girth $g$ and valency $k$. Then $\Sigma$ is called a $(k,g)$-cage
\footnote{Such a graph can be found under different names in the litterature; for example, it is called a $(k,g)$-graph  by Biggs in \cite{Biggs2} and it is called $(k,g)$-Moore graph by Holton and Sheehan in 
\cite{HandSh1993}.} 
if $|V\Sigma|=n_0(k,g)$.
\end{defn}

The following Lemma~\ref{basic} is a reduction result which plays an important
role to prove our main theorem.

\begin{lem}\label{basic}
 Let $\Si$ be a connected graph of diameter $d$, valency $k\geq 3$, girth $g$ and $G\leq \Aut(\Si)$. 
Suppose that $\Ga=\ssig$ is locally $(G,2d)$-distance transitive. Then $\Si$ is a $(k,g)$-cage.  Moreover $g\in \{3,4,5,6, 8,12\}$.
\end{lem}

\begin{proof}
Let $D=\diam(\ssig)$. We have seen in \cite{DDP2009} that $2d\leq D\leq 2d+2$. 
Suppose first that $D=2d$ or $2d+1$. By Lemma \ref{lem:line}, $\Sigma$ and $L(\Sigma)$ are $G$-distance transitive, and so by \cite{Biggs2},  $\Sigma$ is a
$(k,g)$-cage. 

Now suppose that $D=2d+2$. Then $\Ga$ contains two vertices at distance $D$, and since $D$ is even, they are in the same bipart of $\Gamma$. 
Since $d_\Ga(u,v)=2 d_\Si(u,v)\leq 2d$ for $u,v\in V\Si$, we must have two edges    $e=\{x,y\}, \ f=\{u,v\}\in E\Si$ such that $d_\Ga(e,f)=2d+2$.
We can see easily that therefore $d_{\Si}(x,u)=d_{\Si}(x,v)=d_{\Si}(y,u)=d_{\Si}(y,v)=d$. 

By Lemma \ref{lem:d-arc}, we know that $2d\leq g\leq 2d+1$.  Suppose that $g=2d$. Then  there exists $e'=\{x',y'\}\in \Gamma_{2d}(e)$ such that $d_{\Gamma}(e,x')=d_{\Gamma}(e,y')=2d-1$. 
There also exists $f'=\{w,u\}\in \Gamma_{2d}(e)$ such that $d_{\Gamma}(e,u)=2d+1$ and $d_{\Gamma}(e,w)=2d-1$. Note that $G_{e}$ cannot map $f'$ to $e'$, and so we get a contradiction to the fact that
$\Ga$ is locally $(G,2d)$-distance transitive. 
Thus $g=2d+1$. Moreover all edges in $\Gamma_{2d}(e)$ have one vertex in $\Gamma_{2d-1}(e)$ and one in $\Gamma_{2d+1}(e)$.

All the vertices of $\Si$ are at an odd distance (in $\Ga$) from $e$, and so $|V\Si|=|\Gamma_{1}(e)|+|\Gamma_{3}(e)|+\ldots |\Gamma_{2d+1}(e)|$. 
It follows  easily from $g=2d+1$  that $|\Gamma_{2i-1}(e)|=2(k-1)^{i-1}$ for $1\leq i\leq d$. We now want to determine $|\Gamma_{2d+1}(e)|$.

Recall that  $u\in \Gamma_{2d+1}(e)$. Since $g=2d+1$ and we have $d_{\Si}(u,x)=d_{\Si}(u,y)=d$, there exists a $d$-geodesic of $\Si$ from $u$ to $x$ and a $d$-geodesic of $\Si$ from $u$ to $y$. 
 If these two $d$-geodesics shared the same first edge, then $\Si$ would contain a cycle of length less than $2d-1$, a contradiction. 
So $u$ is contained in at least  two edges in $\Gamma_{2d}(e)$.
If $u$ was on at least three edges in $\Gamma_{2d}(e)$, then there would be at least two $d$-geodesics  of $\Si$  from $u$ to say $x$, by the pidgeon-hole principle, 
and so there would be a cycle of length at most $2d$ in $\Si$, a contradiction. 
So $u$ is on exactly two edges in $\Gamma_{2d}(e)$ and  on $k-2$ edges in  $\Gamma_{2d+1}(e)$. Notice also that since $g=2d+1$, there is no edges in between vertices of $\Gamma_{2d-1}(e)$. 
To determine  $|\Gamma_{2d+1}(e)|$, we now count the number of edges in $\Gamma_{2d}(e)$ in two ways. 
We recall that all such edges have one vertex in $\Gamma_{2d-1}(e)$ and one in  $\Gamma_{2d+1}(e)$. Hence $$|\Gamma_{2d-1}(e)|\cdot (k-1)=|\Gamma_{2d+1}(e)|\cdot 2.$$ 
So $|\Gamma_{2d+1}(e)|=(k-1)^d$.  

Therefore $|V\Si|=2+2(k-1)+\cdots +2(k-1)^{d-1}+(k-1)^d=1+k+k(k-1)+\cdots +k(k-1)^{d-1}=1+\sum_{i=0}^{d-1} k(k-1)^{i}$. This equality is easily shown by induction. 
Since $g=2d+1$, we have $|V\Si|=1+\sum_{i=0}^{\frac{g-3}{2}}k(k-1)^{i}=n_0(k,g)$.
Therefore $\Si$ is a $(k,g)$-cage. Moreover by \cite[pp. 189--190]{HandSh1993},  $g\in \{3,4,5,6, 8, 12\}$.
\end{proof}

When $\Sigma$ is bipartite, the diameter of $\ssig$ is determined. 
\begin{lem}\label{doublediam}
Let $\Sigma$ be a bipartite graph and $\Ga$ be its subdivision graph, denoted by $S(\Sigma)$. Then $\diam(\Ga)=2\,\diam(\Sigma)$.
\end{lem}
\begin{proof}
Note that in a bipartite graph $\Sigma$, if $\{x,y\}$ is an edge, then for any vertex $v$, $|\dist_\Sigma(x,v)-\dist_\Sigma(y,v)|=1$.
For $u,v\in V\Sigma$, it is obvious that  $\dist_\Ga(u,v)=2\,\dist_\Sigma(u,v)\leq 2\,\diam(\Sigma)$.
Let  $u\in V\Sigma$ and $e=\{x,y\}\in E\Sigma$. Then  $\dist_\Ga(u,e)=2\,\min\{\dist_\Sigma(u,x),\dist_\Sigma(u,y)\}+1$. Since $|\dist_\Sigma(u,x)-\dist_\Sigma(u,y)|=1$, we have that $\min\{\dist_\Sigma(u,x),\dist_\Sigma(u,y)\}\leq \diam(\Sigma)-1$, and so $\dist_\Ga(u,e)\leq 2\,\diam(\Sigma)-1$.
Finally, let $e,e'\in E\Sigma$ where $e=\{x,y\}$ and $e'=\{x',y'\}$. Then $\dist_\Ga(e,f)=2\,\min\{\dist_\Sigma(x,x'),\dist_\Sigma(x,y'),\dist_\Sigma(y,x'),\dist_\Sigma(y,y')\}+2$. As before that minimum is less than or equal to $\diam(\Sigma)-1$, and so $\dist_\Ga(e,f)\leq 2\,\diam(\Sigma)$.
Hence $\diam(\Ga)\leq 2\,\diam(\Sigma)$. There exist  $u,v\in V\Sigma$ with $\dist_\Sigma(u,v)=\diam(\Sigma)$, so $\dist_\Ga(u,v)=2\,\diam(\Sigma)$. Thus  $\diam(\Ga)=2\,\diam(\Sigma)$.
\end{proof}

Finally, we will need the following Lemma about $(k,6)$-cages.
\begin{lem}\label{PropDes}
Let $\Sigma$ be finite, connected, undirected graph and $G\leq \Aut(\Si)$. If $\Si$ is a $G$-distance transitive $(k,6)$-cage, then 
$\Si$  is the incidence graph of a Desarguesian projective plane of order $k-1$, where $k-1$ is a prime power $q$ and $\PSL(3,q)\leq G\leq \Aut(\PSL(3,q))$.
\end{lem}
\begin{proof}
 Since $\Si$ is a $(k,6)$-cage, by \cite{Singleton1966}, $\Si$ is the incidence graph of a projective plane $\calP$ of order $k-1$.  
For all points $p,q, p',q'\in \calP$, we have  $d_{\Si}(p,q)=d_{\Si}(p',q')=2$ and since $\Si$ is $G$-distance transitive, there exists $\alpha\in G$ such that $(p,q)^{\alpha}=(p',q')$. 
Hence $G$ is $2$-transitive on the points of $\calP$. Therefore, by the  Ostrom-Wagner Theorem \cite{OW}, $\calP$ is Desarguesian, $k-1$ is a prime power $q$, 
and $\PSL(3,q)\leq G\leq \Aut(\PSL(3,q))$. 
\end{proof}

\section{Proof of main theorem}\label{proof}

In this section, we prove Theorems~\ref{char}.

In what follows, suppose that $\Sigma$ is a connected graph with diameter $d\geq 2$, girth $g$ and $G\leq
\Aut(\Sigma)$. Also let
 $\ssig$ be the subdivision graph of $\Si$. Set $D:=\diam(\ssig)$.

\subsection{(a)$\Longrightarrow$ (b)} Suppose that $\ssig$ is locally $(G,2d)$-distance transitive. By Lemma \ref{lem:line}, $G$ is transitive on the vertices of $\Si$, and so $\Si$ is regular,
with valency $k$.

Suppose first that $\Si$ has valency $k=2$. Since $\Si$ is connected, $\Si$ is a cycle $C_n$ ($n\geq 3$), which has diameter $d=\lfloor n/2\rfloor$ and girth $n$. We have that $\ssig$ is a cycle $C_{2n}$.  
Since $\ssig$ is locally $(G,2d)$-distance transitive, it is easy to see that $G$ must be $D_{2n}=\Aut(\Si)$. 

From now on we assume that the valency $k$ is at least $3$.  By Lemma \ref{basic},  we have that $\Si$ is a $(k,g)$-cage and $g\in \{3,4,5,6,8,12\}$. 
By Lemma \ref{lem:d-arc}, if $g$ is even, then $d=g/2$, and if $g$ is odd, then $d=(g-1)/2$. Moreover, $g\geq 2d\geq 4$. 

Suppose $g=4$. Then $d=2$. Since $\ssig$ is locally $(G,4)$-distance transitive, this case was covered in \cite[Theorem 1.4]{DDP2009}, where we proved that $\Sigma=K_{n,n}$ and 
$G$ is a subgroup of $S_{n}\wr S_{2}$ satisfying Condition (*), see the introduction. Since $k\geq 3$, we have $n\geq 3$. Since $\Si$ is bipartite, $D=4$ by Lemma \ref{doublediam}.

Suppose $g=5$. Then $d=2$ and  $\Si$ is a $(k,5)$-cage.  By \cite{Hoff&Sing}, $k=3$ or $7$.  If $k=3$, then $\Sigma$ is the Petersen graph. 
We have seen in \cite{DDP2009} that if the subdivision graph of the Petersen graph is locally $(G,4)$-distance transitive,  then $G=S_5$. If $k=7$, then $\Sigma$ is the Hoffman-Singleton graph $\HoSi$. 
We have seen in \cite{DDP2009} that if the subdivision graph of $\HoSi$ is locally $(G,4)$-distance transitive,  then $G=PSU(3,5)$ or $ PSU(3,5):2$.

Now suppose $g\in\{6,8,12\}$. Then $d=g/2$, and by Lemma \ref{lem:d-arc} $\Si$ is $(G,s)$-transitive for some $s\geq d$. Therefore $g=2d\leq 2s$.
If $g=8,12$, we know $s\geq 4$. If $g=6$, suppose $s\geq 4$. Then  we can use Theorem \ref{Weiss}. 
In particular $d\leq s\leq \frac{g+2}{2}=d+1$. Only the first four lines of Table \ref{tab:Weiss} can happen here, so $\Si$ is the incidence graph of a generalised polygon. 
It follows then that $g=2d$. Moreover, since $\Si$ is bipartite, $D=2d$ by Lemma  \ref{doublediam}. We also know that $G$ contains 
a duality, since $G$ is transitive on $V\Si$.  If we are in Line 1, 2, or 4 of Table \ref{tab:Weiss}, we are done.

Suppose we are in the case of Line 3, that is $\Si=\Inc(W(3,2))$, $s=4$ and $G'=A_6$. Then  $G=S_6, \PSp(4,2), \PGL(2,9)$ or $M_{10}$. The first two cases cannot happen because they don't contain a duality 
of $W(3,2)$. 
Consider the case of $G=\PGL(2,9)$. We claim that $\PGL(2,9)_{e}$ has two orbits on  $\ssig_{8}(e)$, for $e$ an edge of $\ssig$, and consequently $\ssig$ is not locally $(PGL(2,9),8)$-distance transitive.. 
Note that $e$ represents a chamber (that is, a point-line flag) of the generalized quadrangle $W(3,2)$, and $\ssig_{8}(e)$ consists of the chambers opposite $e$ 
(in the sense of the generalised polygons/buildings). The action of $\PGL(2,9)$ on chambers can be seen as the action on the 45 unordered pairs of the projective line $\PG(1,9)$, with two pairs 
$\{a,b\}$ and $\{c,d\}$ representing opposite chambers if they are disjoint and the cross-ratio $(a,b;c,d)$ is a non-square in $\GF(9)$, see \cite{GoVM}. then it is straightforward to see that $G_e$,
that is, the stabiliser of a pair in  $\PGL(2,9)$ has order 16 and two orbits of size 8 on opposite chambers, that is, $\ssig_{8}(e)$. Notice that the action of $A_6\cong \PSL(2,9)$ on chambers of $W(3,2)$ can
be described similarly. In this case the stabiliser of an unordered pair has order 8 and is semi-regular, so  $\PGL(2,9)_{e}$ and  $\PSL(2,9)_{e}$ have the same two orbits on  $\ssig_{8}(e)$.
Hence we must have $G=M_{10}$. 

The only remaining case is if  $g=6$, and $\Si$ is $(G,3)$-transitive.
Note that $\Si$ is a $(k,6)$-cage. By Lemma \ref{lem:line}, $\Si$ is $G$-distance transitive, and so by Lemma \ref{PropDes}, $\Sigma$ is the incidence graph of a Desarguesian projective plane 
of order $q=k-1$ and   $\PSL(3,q)\leq G\leq \Aut(\PSL(3,q))$.  But, by \cite{Li4arc}, $\Sigma$ is $(G,4)$-arc transitive in this case, a contradiction.

\subsection{(b)$\Longrightarrow$ (c)}
For the first three lines of Table \ref{table1}, we have proved in \cite[ Examples 5.5, 5.6 and 5.3]{DDP2009} that  $\ssig$ is locally $G$-distance transitive.

Suppose we are in the case of Line 4, 5, or 7 of Table \ref{table1}. 
Then, by  \cite{Li4arc}, $\Si$ is $(G,d+1)$-arc transitive. Hence, by  \cite[Theorem 1.2]{DDP2009}, $\ssig$ is locally $(G,2d+1)$-arc transitive. 
Since $D=2d$, $\ssig$ is locally $G$-distance transitive. 

Suppose we are in the case of Line 6, so $D=8$.  We know that $\Si$ is $(G,4)$-arc transitive, and so, by \cite[Theorem 1.3]{DDP2009}, $\ssig$ is locally $(G,7)$-distance transitive. 
Let $v\in V\Sigma$ and  $e\in E\Sigma$.
We have  $\ssig_{8}(v)=\Si_4(v)$, and so $G=M_{10}$ is transitive on  $\ssig_{8}(v)$, since  $\Si$ is $(G,4)$-arc transitive. 
So if we prove that  $G_{e}$ is transitive on $\ssig_{8}(e)$, then we will have that $\ssig$ is locally $G$-distance transitive. 
Suppose that $G_{e}$ is not transitive on $\ssig_{8}(e)$. We have seen above that the subgroup $\PSL(2,9)_e$ of $G_e$ has two orbits of size 8 on $\ssig_{8}(e)$, 
so we must have that $G_e$ has those same two orbits. We have also seen that $\PGL(2,9)_e$ has those same two orbits.
We know that $\langle G,\PGL(2,9)\rangle\cong \PGaL(2,9)\cong\Aut(\PSp(4,2))$, and so $\PGaL(2,9)_e$ also has those same two orbits on $\ssig_{8}(e)$. This means that $\ssig$ is not 
locally $\Aut(\PSp(4,2))$-distance transitive, which contradicts the previous paragraph. Thus we have that  $G_{e}$ is transitive on $\ssig_{8}(e)$.

It is obvious that if $\Si=C_n$ and $G=D_{2n}$ (Line 8), then $\ssig$ is locally $G$-distance transitive. 

\subsection{(c)$\Longrightarrow$ (a)} Since $D\geq 2d$, this implication is trivial.

\section{Acknowledgments}
The authors thank Hendrik Van Maldeghem for the argument regarding the transitivity of an edge stabiliser in $M_{10}$ and $\PGL(2,9)$. 



\begin{thebibliography}{1}
 


\bibliographystyle{amsplain}
\bibitem{Biggs2}
Norman Biggs, \emph{The symmetry of line graphs}, Utilitas Math. \textbf{5}
  (1974), 113--121. 

\bibitem{DDP2009}
Daneshkhah, A., Devillers A. Praeger~C.E., \emph{Symmetry properties of
  subdivision graphs}, submitted.

\bibitem{GoVM}
E. Govaert, H. Van Maldeghem, \emph{Distance-preserving maps in generalized polygons. I. Maps on flags},
Beiträge Algebra Geom. \textbf{43} (2002), no.~1, 89--110. 

\bibitem{Hoff&Sing}
A.~J. Hoffman and R.~R. Singleton, \emph{On {M}oore graphs with diameters {$2$}
  and {$3$}}, IBM J. Res. Develop. \textbf{4} (1960), 497--504. 
\bibitem{HandSh1993}
D.~A. Holton and J.~Sheehan, \emph{The {P}etersen graph}, Australian
  Mathematical Society Lecture Series, vol.~7, Cambridge University Press,
  Cambridge, 1993. 

\bibitem{Li4arc}
Cai~Heng Li, \emph{The finite vertex-primitive and vertex-biprimitive
  {$s$}-transitive graphs for {$s\geq 4$}}, Trans. Amer. Math. Soc.
  \textbf{353} (2001), no.~9, 3511--3529 (electronic).  

\bibitem{OW}
T. Ostrom, A. Wagner, \emph{On projective and affine planes with transitive collineation groups}, Math. Z \textbf{71} (1959), 186--199. 

\bibitem{Singleton1966}
Robert Singleton, \emph{On minimal graphs of maximum even girth}, J.
  Combinatorial Theory \textbf{1} (1966), 306--332. 

\bibitem{Weiss85}
Richard Weiss, \emph{Distance-transitive graphs and generalized polygons},
  Arch. Math. (Basel) \textbf{45} (1985), no.~2, 186--192. 
\end{thebibliography}
\end{document}